\documentclass[a4paper]{amsart}
\usepackage{amsmath,amsthm,amssymb}
\usepackage{mathtools}
\usepackage{scalerel}
\usepackage[shortlabels]{enumitem}
\usepackage{thmtools,thm-restate}
\makeatletter
\@ifundefined{newcounteralias}{}{%
  \renewcommand\thmt@autorefsetup{\@xa\def\csname\thmt@envname autorefname\@xa\endcsname\@xa{\thmt@thmname}}%
}
\makeatother
\usepackage{todonotes}
\usepackage{hyperref}
\usepackage[capitalize]{cleveref}
\usepackage{orcidlink}

\usepackage[
  hyperref=auto,
  mincrossrefs=999,
  backend=biber,
 style=numeric,
 giveninits
]{biblatex}

\DeclareMathOperator{\Aut}{Aut}

\DeclareMathOperator{\Class}{Class}

\DeclareMathOperator{\Coin}{Coin}
\DeclareMathOperator{\coker}{coker}
\DeclareMathOperator{\End}{End}

\DeclareMathOperator{\Fix}{Fix}
\DeclareMathOperator{\GL}{GL}
\DeclareMathOperator{\Hom}{Hom}
\DeclareMathOperator{\Id}{Id}
\DeclareMathOperator{\im}{Im}
\DeclareMathOperator{\Inn}{Inn}
\DeclareMathOperator{\Irr}{Irr}

\DeclareMathOperator{\SpecR}{Spec_{R}}

\DeclareMathOperator{\Stab}{Stab}
\DeclareMathOperator{\Tr}{Tr}
\newcommand*{\B}{\mathcal{B}}
\newcommand*{\C}{\mathbb{C}}
\newcommand*{\card}[1]{\#{#1}}
\newcommand*{\Conj}{\mathcal{C}}
\newcommand*{\conj}{\sim}
\newcommand*{\cconj}[1]{\overline{#1}}

\newcommand*{\dual}[1]{\widehat{#1}}

\newcommand*{\F}{\mathbb{F}}

\newcommand*{\grpgen}[1]{\langle#1 \rangle}
\newcommand*{\grppres}[2]{\langle#1 \mid #2 \rangle}

\newcommand*{\ie}{i.e.\ }

\newcommand*{\ind}[1]{\Delta_{#1}}

\newcommand*{\inprodarg}[2]{\left\langle #1, #2 \right \rangle}
\newcommand*{\inv}[1]{{#1}^{-1}}
\newcommand*{\invrestr}[2]{{#1}|_{#2}^{-1}}

\newcommand*{\N}{\mathbb{N}}

\newcommand*{\Rconj}[1]{\sim_{#1}}
\newcommand*{\Reid}{\mathcal{R}}

\newcommand*{\restr}[2]{{#1}|_{#2}}
\newcommand*{\Rinf}{R_{\infty}}

\newcommand*{\Z}{\mathbb{Z}}

\renewcommand{\phi}{\varphi}

\let\originalleft\left
\let\originalright\right
\renewcommand{\left}{\mathopen{}\mathclose\bgroup\originalleft}
\renewcommand{\right}{\aftergroup\egroup\originalright}

\declaretheorem[style=definition, name = Definition, numberwithin=section]{defin}
\declaretheorem[style=definition, name = Example, sibling=defin]{example}
\declaretheorem[name = Theorem, sibling=defin]{theorem}
\declaretheorem[name = Lemma, sibling=defin]{lemma}
\declaretheorem[name = Proposition, sibling=defin]{prop}
\declaretheorem[name = Corollary, sibling=defin]{cor}
\declaretheorem[style=remark, name = Remark, numbered = no]{remark}

\declaretheorem[name = Theorem, numbered = no]{theorem*}
\declaretheorem[name = Question, sibling=defin]{quest}
\declaretheorem[name = Question, numbered = no]{quest*}

\numberwithin{equation}{section}

\crefname{prop}{Proposition}{Propositions}
\crefname{cor}{Corollary}{Corollaries}
\crefname{lemma}{Lemma}{Lemmata}

\newcommand{\set}[2]{\ensuremath{\{\,{#1}\mid{#2}\,\}}}
\newcommand{\restri}[2]{\ensuremath{{\left.\kern-\nulldelimiterspace{#1}\right|}_{#2}}}

\title{Bi-twisted conjugacy in finite groups}
\author[Pieter Senden]{Pieter Senden\,\orcidlink{0000-0002-3107-6775}}
\author[Sam Tertooy]{Sam Tertooy\,\orcidlink{0000-0002-5750-9153}}
\date{\today}
\address{KU Leuven Campus Kulak Kortrijk\\
	E.~Sabbelaan 53\\
	8500 Kortrijk\\
	Belgium}
\email{sam.tertooy@kuleuven.be}
\email{pieter.senden@telenet.be}


\begin{document}
\begin{abstract}
	We provide two alternative ways to determine the number of bi-twisted conjugacy classes in a finite group: one using irreducible characters and one using ordinary conjugacy classes.
	In addition, we show various equalities and (sharp) inequalities for Reidemeister numbers, as well as relations between bi-twisted conjugacy, representation theory, and fixed-point free automorphisms.
\end{abstract}

\maketitle
\section{Introduction}
	Let \(G\) be a group and \(\phi, \psi \in \End(G)\).
	We say that two elements \(x, y \in G\) are \emph{\((\phi, \psi)\)-conjugate} if \(x = \phi(z) y \inv{\psi(z)}\) for some \(z \in G\), in which case we write \(x \Rconj{\phi, \psi} y\).
	We let \([x]_{\phi, \psi}\) denote the \((\phi, \psi)\)-conjugacy class of \(x\) and write \(\Reid[\phi, \psi]\) for the set of all \((\phi, \psi)\)-conjugacy classes.
	The number of \((\phi, \psi)\)-conjugacy classes is called the \emph{Reidemeister number} of \((\phi, \psi)\) and is denoted by \(R(\phi, \psi)\).
	In general, we speak of \emph{bi-twisted conjugacy}.
	
	Equivalently, we can view bi-twisted conjugacy as a group action,
	\[
		G \times G \to G \colon (g, h) \mapsto g \cdot h \coloneq \phi(g) h \inv{\psi(g)}
	\]
	for which we let \(\Stab_{\phi, \psi}(x)\) denote the \((\phi, \psi)\)-stabiliser of \(x \in G\), \ie
	\[
		\Stab_{\phi, \psi}(x) \coloneq \set{g \in G}{\phi(g)x\inv{\psi(g)} = x}.
	\]
	
Bi-twisted conjugacy has been studied in various contexts and in various degrees of generalisation: A.\ Fel'shtyn and B.\ Klopsch, as well as V.\ Roman'kov, studied it for nilpotent groups \cite{FelshtynKlopsch22,Romankov16}, and the second author studied it for twisted conjugacy separability with extension to \(\phi, \psi \in \Hom(H, G)\) for some other group \(H\) \cite{Tertooy25}.

For \(\phi = \Id\), we get (single) twisted conjugacy with Reidemeister number \(R(\psi) \coloneq R(\Id, \psi)\).
The notation then simplifies to \(x \Rconj{\psi} y\), \([x]_{\psi}\), \(\Reid[\psi]\) and \(\Stab_{\psi}(x)\).
In twisted conjugacy, an additional concept is often studied, namely the \emph{Reidemeister spectrum} of \(G\), given by \(\SpecR(G) \coloneq \set{R(\psi)}{\psi \in \Aut(G)}\).

Twisted conjugacy in finite groups arises in various places.
T.\ Shintani \cite{Shintani76} showed that the \(F\)-invariant irreducible representations of \(\GL(n, q^{m})\) are in bijective correspondence with the irreducible representations of \(\GL(n, q)\), where \(n \geq 1\), \(q\) is a prime power and \(F \colon \GL(n, q^{m}) \to \GL(n, q^m) \colon (a_{ij})_{ij} \mapsto (a_{ij}^{q})_{ij}\) is the Frobenius map.
He proved this using a method now known as Shintani descent, which concerns \(F\)-twisted conjugacy in \(\GL(n, q^{m})\).
Several generalisations of Shintani's result have been proven since \cite{Shoji85,Shoji87,Shoji92,Deshpande16}.

In linear algebraic groups, P.\ Deligne and G.\ Lusztig parameterised the \(\F_{q}\)-rational conjugacy classes of maximal tori by \(F\)-conjugacy classes in the Weyl group \cite{DeligneLusztig76}, with similar results following later \cite{Lusztig14,Chen22}.

In Nielsen fixed-point theory, which focuses on determining lower bounds on the number of fixed points of a continuous self-map on a topological space, twisted conjugacy is largely studied in infinite groups.
However, one of the scarce results on finite groups in Nielsen theory is the following due to B.\ Jiang \cite{Jiang83}, in the context of fixed-point theory on compact polyhedra with finite fundamental group:
\begin{prop}
	Let \(A\) be a finite abelian group and \(\psi \in \End(A)\).
	Then \(R(\psi) = \card{\coker(\Id - \psi)}\).
\end{prop}

In most of the aforementioned research, twisted conjugacy is used as a tool to answer questions.
Twisted conjugacy in finite groups in and of itself has been studied as well, with one of the first results due to I.\ Ado and R.\ Ree:
\begin{theorem}	\label{theo:AdoReeNagao}
	Let \(G\) be a finite group and \(\psi \in \End(G)\).
	The following quantities are equal:
	\begin{enumerate}[(1)]
		\item The number of \(\psi\)-conjugacy classes, \ie \(R(\psi)\);
		\item The number of \(\psi\)-invariant ordinary conjugacy classes;
		\item The number of \(\psi\)-invariant irreducible complex characters.
	\end{enumerate}
\end{theorem}
Ado proved equality of (1) and (3) for automorphisms \cite{Ado55}, which was then generalised to endomorphisms by Ree \cite{Ree59}.
In the same paper, Ree proved equality of (2) and (3), from which equality of (1) and (2) follows.
H.\ Nagao finally gave a short and direct proof of equality of (1) and (2) \cite{Nagao62}.

The equality of (1) and (2) was rediscovered by A.\ Fel'shtyn and R.\ Hill in \cite[Theorem 5]{FelshtynHill94}, where they used it to show the rationality of the Reidemeister zeta function of an endomorphism of a finite group.
The equality of (1) and (3) was rediscovered by the first author in \cite[Theorem~8.1.7]{Senden23}.

More recently, V.\ Bardakov, T.\ Nasybullov and M.\ Neshchadim \cite{BardakovNasybullovNeshchadim13}, D.\ Gon\c{c}alves and Nasybullov \cite{GoncalvesNasybullov18} and C.\ Nicotera \cite{Nicotera24} have studied (finite) groups where the twisted conjugacy class of the unit element is a subgroup. On the level of Reidemeister spectra, the first author has determined the Reidemeister spectrum of finite abelian groups \cite{Senden23a} and split-metacyclic groups \cite{Senden21a}, whereas the second author has proposed and discussed research questions that serve as analogues of the \(\Rinf\)-property and full Reidemeister spectrum for finite groups \cite{Tertooy25a}.

Finally, the research in (bi\nobreakdash-)twisted conjugacy in finite groups ties in with fixed points of endomorphisms, as shown by the following equivalence:

\begin{prop}[{\cite[Proposition~2.2]{Tertooy25}}]\label{prop:ReidemeisterNumber1}
	Let \(G\) be a finite group and \(\psi \in \End(G)\).
Then \(R(\psi) = 1\) if and only if \(\Fix(\psi) = 1\).
\end{prop}
Here, \(\Fix(\psi)\) is the set of fixed points of \(\psi\), that is,
\[
	\Fix(\psi) \coloneq \{g \in G \mid g = \psi(g)\}.
\]
\begin{proof}
	By the orbit-stabiliser theorem, \(\card{[1]_{\psi}} \cdot \card{\Stab_{\psi}(1)} = \card{G}\).
	Since
	\[
		\Stab_{\psi}(1) = \set{g \in G}{g\inv{\psi(g)} = 1} = \set{g \in G}{g = \psi(g)} = \Fix(\psi),
	\]
	we conclude that \(\Fix(\psi) = 1\) if and only if \(\card{[1]_{\psi}} = \card{G}\).
	The latter is equivalent with \(R(\psi) = 1\).
\end{proof}

In other words, determining which finite groups admit an automorphism with Reidemeister number equal to \(1\) is equivalent with determining which finite groups admit an automorphism with no non-trivial fixed points.
Such an automorphism is called \emph{fixed-point free}.
One of the major results on this topic is due to P.\ Rowley:
\begin{theorem}[{\cite{Rowley95}}]
	Let \(G\) be a finite group.
If \(G\) admits a fixed-point free automorphism, then \(G\) is solvable.	
\end{theorem}

For more information on groups with fixed-point free automorphisms, we refer the reader to \cite{GorensteinHerstein61,ShumyatskyTamarozzi02,Thompson59}.

The goal of this paper is to generalise \cref{theo:AdoReeNagao}.
Firstly, we generalise equality of (1) and (3) to bi-twisted conjugacy:
\begin{restatable*}{cor}{countingCharacters}	\label{cor:SumInproductFormulaReidemeisterNumber}
	Let \(G\) be a finite group and \(\phi, \psi \in \End(G)\). Then
	\[
		R(\phi, \psi) = \sum_{\chi \in \Irr(G)} \inprodarg{\chi \circ \phi}{\chi \circ \psi}.
	\]
\end{restatable*}
Here, \(\Irr(G)\) denotes the set of irreducible complex characters of \(G\).

Secondly, we generalise the equality of (1) and (2) in two ways.
One generalisation is to bi-twisted conjugacy:
\begin{restatable*}{prop}{biTwistedConjugacyNumberInTermsOfConjugacyClasses}	\label{prop:biTwistedConjugacyNumberInTermsOfConjugacyClasses}
	Let \(G\) be a finite group and \(\phi, \psi \in \End(G)\).
	Then
	\[
		R(\phi, \psi) = \sum_{[g] \in \Conj(G)} \frac{\card{\left(\inv{\phi}([g]) \cap \inv{\psi}([g])\right)}}{\card{[g]}}.
	\]
\end{restatable*}

Here, \(\Conj(G)\) denotes the set of (ordinary) conjugacy classes of \(G\).
The other focuses on single twisted conjugacy, and consists of determining Reidemeister numbers by counting fixed points of actions on other twisted conjugacy classes:
\begin{restatable*}{theorem}{reidemeisterNumberEqualsNumberFixedConjugacyClasses}\label{theo:ReidemeisterNumberEqualsNumberFixedConjugacyClasses}
	Let \(G\) be a finite group and \(\phi, \psi \in \End(G)\).
Suppose that \(\phi \circ \psi = \psi \circ \phi\) and that \(g \Rconj{\psi} \phi(g)\) for all \(g \in G\).
Then the map
	\[
		\Xi\colon \Reid[\phi] \to \Reid[\phi]\colon [g]_{\phi} \mapsto [\psi(g)]_{\phi}
	\]
	is well defined and \(R(\psi) = \card{\Fix(\Xi)}\).
	
	In particular, the result holds for \(\phi = \Id\) and \(\phi = \psi^{k}\) for any \(k \geq 1\).
\end{restatable*}

Finally, we prove a sharp upper bound on the highest non-trivial number of bi-twisted conjugacy classes in a finite non-abelian group:
\begin{restatable*}{prop}{upperboundBiTwistedConjugacyNumbers}
\label{prop:upperboundBiTwistedConjugacyNumbers}
    Let \(G\) be a non-abelian finite group and \(\phi,\psi \in \End(G)\). Let \(p\) be the smallest prime divisor of \(\card{G}\), and let \(q\) be either the second smallest prime divisor (if \(p^3 \nmid \card{G}\)) or equal to \(p\) (if \(p^3 \mid \card{G}\)). The following are equivalent:
    \begin{enumerate}[(1)]
    	\item \(R(\phi,\psi) > \frac{p+q-1}{pq} \card{G}\);
    	\item \(\phi = \psi\) and \(\im(\phi) \leq Z(G)\);
    	\item \(R(\phi,\psi) = \card{G}\).
    \end{enumerate}
\end{restatable*}

\section{Counting complex characters}
\subsection{Formulae}

Since we can view bi-twisted conjugacy as a group action, this yields, for every \(g \in G\), the map \(T_{\phi, \psi}(g)\colon G \to G\colon h \mapsto \phi(g) h \inv{\psi(g)}\), which is a permutation of \(G\).
Hence, we get a permutation representation \(T_{\phi, \psi}\) of \(G\), which we call the \emph{bi-twisted conjugation representation}.
If \(\phi = \Id\), we also simply write \(T_{\psi}\).
For \(\phi = \psi = \Id\), we obtain the classical conjugation representation, which has been intensively studied in the literature \cite{Formanek71,Frame47,Roth71,Solomon61}.

We generalise some of the results by R.\ Roth \cite{Roth71} to the bi-twisted conjugation representation and use these to determine an alternative way to express \(R(\phi, \psi)\) in terms of the irreducible characters of \(G\) (\cref{cor:SumInproductFormulaReidemeisterNumber}).

\begin{lemma}	\label{lem:ExpressionInproductBitwistedRepresentationIrreducibleCharacter}
	Let \(G\) be a finite group, \(\phi, \psi \in \End(G)\) and let \(\theta_{\phi, \psi}\) be the character of the \((\phi,\psi)\)-conjugacy representation. For every \(\chi \in \Irr(G)\), we have
	\[
		\inprodarg{\theta_{\phi, \psi}}{\chi} = \sum_{\substack{g \in G \\ [\phi(g)] = [\psi(g)]}} \frac{1}{\card{[\phi(g)]}} \cconj{\chi}(g).
	\]
\end{lemma}
\begin{proof}
	Let \(g \in G\) be arbitrary. Since \(T_{\phi, \psi}\) is a permutation representation, \(\theta_{\phi, \psi}(g) = \card{\Fix(T_{\phi, \psi}(g))}\). We further rewrite this:
	\begin{align*}
		\theta_{\phi, \psi}(g)	&=	\card{\Fix(T_{\phi, \psi}(g))}	\\
								&=	\card{\set{x \in G}{\phi(g) x \inv{\psi(g)} = x}}	\\
								&=	\card{\set{x \in G}{\phi(g)^{x} = \psi(g)}}.
	\end{align*}
	If \(\phi(g)\) and \(\psi(g)\) are conjugate, this equals \(\card{C_{G}(\phi(g))}\); otherwise, this equals \(0\). Therefore, for \(\chi \in \Irr(G)\) arbitrary, we find
	\begin{align*}
		\inprodarg{\theta_{\phi, \psi}}{\chi}	&=	\frac{1}{\card{G}} \sum_{g \in G} \theta_{\phi, \psi}(g) \cconj{\chi}(g)	\\
												&=	\sum_{\substack{g \in G \\ [\phi(g)] = [\psi(g)]}} \frac{\card{C_{G}(\phi(g))}}{\card{G}} \cconj{\chi}(g)	\\
												&=	\sum_{\substack{g \in G \\ [\phi(g)] = [\psi(g)]}} \frac{1}{\card{[\phi(g)]}} \cconj{\chi}(g),
	\end{align*}
	which finishes the proof.
\end{proof}

\begin{theorem}	\label{theo:SumFormulaCharacterBitwistedRepresentation}
	Let \(G\) be a finite group, \(\phi, \psi \in \End(G)\) and let \(\theta_{\phi, \psi}\) be the character of the \((\phi,\psi)\)-conjugacy representation. Then
	\begin{equation}	\label{eq:SumFormulaCharacterBitwistedRepresentation}
		\theta_{\phi, \psi} = \sum_{\chi \in \Irr(G)} (\chi \circ \phi)(\cconj{\chi \circ \psi}).
	\end{equation}
\end{theorem}
\begin{proof}
	Let \(\omega\) denote the right-hand side of \eqref{eq:SumFormulaCharacterBitwistedRepresentation}. Let \(\Irr(G) = \{\chi_{1}, \ldots, \chi_{c}\}\), with \(c\) the number of irreducible representations of \(G\). Let \(j \in \{1, \ldots, c\}\) be arbitrary. Then
	\begin{align*}
		\inprodarg{\omega}{\chi_{j}}	&=	\frac{1}{\card{G}} \sum_{g \in G} \sum_{i = 1}^{c} (\chi_{i} \circ \phi)(g) (\cconj{\chi_{i} \circ \psi})(g) \cconj{\chi_{j}}(g)	\\
										&=	\frac{1}{\card{G}} \sum_{g \in G} \sum_{i = 1}^{c} \chi_{i}(\phi(g)) \cconj{\chi_{i}}(\psi(g))\cconj{\chi_{j}}(g)	\\
										&=	\frac{1}{\card{G}} \sum_{g \in G} \cconj{\chi_{j}}(g) \sum_{i = 1}^{c} \chi_{i}(\phi(g)) \cconj{\chi_{i}}(\psi(g)).
	\end{align*}
	For \(x, y \in G\), the sum \(\sum_{i = 1}^{c} \chi_{i}(x) \cconj{\chi_{i}}(y)\) is the inner product of two columns of the character table of \(G\), namely the columns of the conjugacy classes of \(x\) and \(y\). It is well known (\cite[Chapter~19, Theorem~2.3(iii)]{Karpilovsky92}) that this is equal to \(\card{C_{G}(x)}\) if they are conjugate, and equal to \(0\) if they are not. Hence, we obtain
	\begin{align*}
		\inprodarg{\omega}{\chi_{j}}	&=	\frac{1}{\card{G}} \sum_{\substack{g \in G \\ [\phi(g)] = [\psi(g)]}} \cconj{\chi_{j}}(g) \card{C_{G}(\phi(g))}	\\
										&=	\sum_{\substack{g \in G \\ [\phi(g)] = [\psi(g)]}} \frac{1}{\card{[\phi(g)]}} \cconj{\chi_{j}}(g).
	\end{align*}
	By \cref{lem:ExpressionInproductBitwistedRepresentationIrreducibleCharacter}, this is equal to \(\inprodarg{\theta_{\phi, \psi}}{\chi_{j}}\). As \(j \in \{1, \ldots, c\}\) was arbitrary, it follows that \(\omega = \theta_{\phi, \psi}\).
\end{proof}

\begin{prop}	\label{prop:ReidemeisterNumberInproductBitwistedRepresentation1}
	Let \(G\) be a finite group, \(\phi, \psi \in \End(G)\) and let \(\theta_{\phi, \psi}\) be the character of the \((\phi,\psi)\)-conjugacy representation. Let \(1\) denote the trivial representation of \(G\). Then \(R(\phi, \psi) = \inprodarg{\theta_{\phi, \psi}}{1}\).
\end{prop}
\begin{proof}
	Given a finite group \(F\) and two (complex) representations \(\rho\colon F \to \GL(V)\) and \(\sigma\colon F \to \GL(W)\), it is well known that \(\inprodarg{\chi_{\rho}}{\chi_{\sigma}}\) equals the dimension of \(\Hom_{F}(V, W)\), where \(\chi_{\rho}\) and \(\chi_{\sigma}\) are the characters of the respective representations (\cite[\S7.2, Lemma~2]{Serre77}). Applied to \(T_{\phi, \psi}\) and \(1\), we have to compute
	\[
		\dim_{\C}(\Hom_{G}(\C G, \C)),
	\]
	where \(\C G\) is the group algebra associated to \(G\).
	
	Let \(f\colon \C G \to \C\) be a \(G\)-invariant linear map, \ie \(f(g \cdot v) = g \cdot f(v)\) for all \(g \in G\) and \(v \in \C G\). By definition of both representations, this is equivalent to \(f\) being constant on \((\phi, \psi)\)-conjugacy classes. For every \(g \in G\), let \(\Delta_{[g]_{\phi, \psi}}\colon \C G \to \C\) be defined by
	\[
		\Delta_{[g]_{\phi, \psi}}(h) = \begin{cases}
 					1	&	\mbox{if } g \Rconj{\phi, \psi} h,	\\
 					0	&	\mbox{otherwise,}
 				\end{cases}
	\]
	for every \(h \in G\) and then extended linearly to the whole of \(\C G\).
	By construction, \(\Delta_{[g]_{\phi, \psi}}\) is \(G\)-invariant. Clearly, the set \(\B \coloneq \set{\Delta_{[g]_{\phi, \psi}}}{[g]_{\phi, \psi} \in \Reid[\phi, \psi]}\) is linearly independent. Conversely, given a \(G\)-invariant linear map \(f \colon \C G \to \C\), it is readily verified that
	\[
		f = \sum_{[g]_{\phi, \psi} \in \Reid[\phi, \psi]} f(g) \Delta_{[g]_{\phi, \psi}}.
	\]
	Hence, \(\B\) is also a generating set. Therefore,
	\[
		\inprodarg{\theta_{\phi, \psi}}{1} = \dim_{\C}(\Hom_{G}(\C G, \C)) = \card{\B} = R(\phi, \psi). \qedhere
	\]
\end{proof}

With this, we can generalise equality of (1) and (3) in \cref{theo:AdoReeNagao} to bi-twisted conjugacy.

\countingCharacters
\begin{proof}
	By \cref{prop:ReidemeisterNumberInproductBitwistedRepresentation1}, \(R(\phi, \psi) = \inprodarg{\theta_{\phi, \psi}}{1}\). Using \cref{theo:SumFormulaCharacterBitwistedRepresentation}, we obtain
	\begin{align*}
		R(\phi, \psi)	&=	\inprodarg{\theta_{\phi, \psi}}{1}	\\
						&=	\sum_{\chi \in \Irr(G)} \inprodarg{(\chi \circ \phi)(\cconj{\chi \circ \psi})}{1}	\\
						&=	\sum_{\chi \in \Irr(G)} \inprodarg{\chi \circ \phi}{\chi \circ \psi}	. \qedhere
	\end{align*}
\end{proof}
For \(\phi = \Id\), we get
\[
	R(\psi) = \sum_{\chi \in \Irr(G)} \inprodarg{\chi}{\chi \circ \psi},
\]
which is the equality of (1) and (3) in \cref{theo:AdoReeNagao}.
Indeed, for \(\chi \in \Irr(G)\), the inner product \(\inprodarg{\chi}{\chi \circ \psi}\) is either \(0\) or \(1\), since both characters have the same degree and \(\chi\) is irreducible.
Consequently, \(\inprodarg{\chi}{\chi \circ \psi} = 1\) if and only if they are equal.
The summation above therefore counts the number of \(\psi\)-invariant irreducible complex characters.

The set of irreducible complex characters of a finite group \(G\) is an orthonormal basis of \(\Class(G)\), the vector space of class functions on \(G\).
Using \cref{cor:SumInproductFormulaReidemeisterNumber}, we show that a similar sum formula for \(R(\phi, \psi)\) holds if we replace \(\Irr(G)\) by any orthonormal basis of \(\Class(G)\).

For an endomorphism \(\phi \in \End(G)\), we let \(\dual{\phi} \colon \Class(G) \to \Class(G) \colon f \mapsto f \circ \phi\) denote the induced map on \(\Class(G)\).

\begin{theorem}	\label{theo:ReidemeisterNumberBiTwistedEqualsTraceBilinearForm}
	Let \(G\) be a finite group and \(\phi, \psi \in \End(G)\).
	Let \(\B\) be an orthonormal basis of \(\Class(G)\) (with respect to the standard inner product).
	Then
	\[
		R(\phi, \psi) = \sum_{b \in \B} \inprodarg{\dual{\phi}(b)}{\dual{\psi}(b)}.
	\]
\end{theorem}
\begin{proof}
	Consider the sesquilinear map
	\[
		B \colon \Class(G) \times \Class(G) \to \C \colon (f, g) \mapsto \inprodarg{\dual{\phi}(f)}{\dual{\psi}(g)}.
	\]
	Its matrix representation with respect to the basis \(\B\) is given by
	\[
		M \coloneq \left(\inprodarg{\dual{\phi}(b_{i})}{\dual{\psi}(b_{j})} \right)_{ij},
	\]
	which means that
	\[
		B(f, g) = x^{T} M \bar{y}
	\]
	where \(x\) and \(y\) are the coordinates of \(f\) and \(g\), respectively, with respect to the basis \(\B\), and \(\bar{y}\) means complex conjugate.
	We then define the trace of \(B\) as \(\Tr B \coloneq \Tr M\) and this is independent of the chosen orthonormal basis.
	
	Indeed, for another orthonormal basis \(\B'\), let \(P\) be the matrix of change of bases from \(\B'\) to \(\B\).
	Since both bases are orthonormal, \(P\) is a unitary matrix, \ie \(P^T = \inv{\bar{P}}\).
	With respect to the basis \(\B'\), the matrix representation of \(B\) is given by
	\[
		B(f, g) = x^T P^T M \bar{P} \bar{y},
	\]
	where \(x\) and \(y\) are the coordinates of \(f\) and \(g\), respectively, with respect to the basis \(\B'\).
	Since \(P^{T} = \inv{\bar{P}}\), we obtain
	\[
		\Tr (P^{T} M \bar{P}) = \Tr (\inv{\bar{P}} M \bar{P}) = \Tr M.
	\]
	
	Now, with respect to the basis \(\Irr(G)\), the matrix representation of \(B\) is equal to
	\[
		\left(\inprodarg{\dual{\phi}(\chi_{i})}{\dual{\psi}(\chi_{j})} \right)_{ij}.
	\]
	Thus, we have
	\[
		\Tr B = \sum_{\chi \in \Irr(G)} \inprodarg{\dual{\phi}(\chi)}{\dual{\psi}(\chi)} = R(\phi, \psi)
	\]
	by \cref{cor:SumInproductFormulaReidemeisterNumber}.
	Therefore,
	\[
		R(\phi, \psi) = \Tr B = \sum_{b \in \B} \inprodarg{\dual{\phi}(b)}{\dual{\psi}(b)}. \qedhere
	\]
\end{proof}

\subsection{Further relations with representation theory}

For bi-twisted conjugacy with automorphisms, we can reformulate \cref{cor:SumInproductFormulaReidemeisterNumber} in terms of coincidences. Given two maps \(f, g\colon X \to Y\), the \emph{coincidence set} of \(f\) and \(g\) is
\[
	\Coin(f, g) \coloneq \set{x \in X}{f(x) = g(x)}.
\]
\begin{prop}
	Let \(G\) be a finite group and \(\phi, \psi \in \Aut(G)\). Then \(R(\phi, \psi) = \card{\Coin(\dual{\phi}, \dual{\psi})}\), where we consider
	\[
		\dual{\phi}\colon \Irr(G) \to \Irr(G)\colon \chi \mapsto \chi \circ \phi,
	\]
	and similarly for \(\dual{\psi}\).
	
	In particular, \(R(\psi) = \card{\Fix(\dual{\psi})}\).
\end{prop}
\begin{proof}
	Since \(\phi\) and \(\psi\) are automorphisms, both \(\dual{\phi}\) and \(\dual{\psi}\) are well defined and bijective. Thus, for every \(\chi \in \Irr(G)\), we have
	\[
		\inprodarg{\chi \circ \phi}{\chi \circ \psi} = \begin{cases}
 									1	&	\mbox{if } \chi \circ \phi = \chi \circ \psi,	\\
 									0	&	\mbox{otherwise.}	
 \end{cases}
	\]
	Hence, by \cref{cor:SumInproductFormulaReidemeisterNumber},
	\[
		R(\phi, \psi) = \sum_{\chi \in \Irr(G)} \inprodarg{\chi \circ \phi}{\chi \circ \psi} = \card{\set{\chi \in \Irr(G)}{\dual{\phi}(\chi) = \dual{\psi}(\chi)}},
	\]
	which is exactly \(\card{\Coin(\dual{\phi}, \dual{\psi})}\).
	
	If \(\phi = \Id\), then \(\dual{\phi} = \Id\) as well and \(\Coin(\dual{\phi}, \dual{\psi}) = \Fix(\dual{\psi})\).
\end{proof}

\begin{remark}
    Given \(\phi, \psi \in \Aut(G)\), it is readily verified that \(R(\phi, \psi) = R(\inv{\phi} \circ \psi)\) (see also \cref{prop:Reidbijections}).
    This essentially reduces bi-twisted conjugacy for automorphisms to twisted conjugacy of a single automorphism.
    Nonetheless, the formulation in terms of coincidences shows how the bi-twisted case is related to representation theory as well.
\end{remark}

\begin{cor}
	Let \(G\) be a finite group and \(\psi \in \Aut(G)\). Then \(R(\psi) = 1\) if and only if \(\dual{\psi}\) only fixes the trivial character.	
\end{cor}
\begin{proof}
	Since \(R(\psi) = \card{\Fix(\dual{\psi})}\) and the trivial character is always fixed by \(\dual{\psi}\), the equivalence follows.
\end{proof}

With this equivalence, we can derive some information about the irreducible representations of a group that does admit fixed-point free automorphisms.
\begin{cor}
	Let \(G\) be a finite group.
	Suppose that \(G\) has either exactly two irreducible representations of degree \(1\), or a unique irreducible representation of degree \(d\) for some \(d \geq 2\).
	Then \(G\) does not admit a fixed-point free automorphism.
\end{cor}
\begin{proof}
	Let \(\rho\) be either the unique non-trivial irreducible representation of degree \(1\), or the unique representation of degree \(d \geq 2\), and let \(\chi\) be its character.
	Let \(\psi \in \Aut(G)\).
    The trivial character \(1\) satisfies \(1 \circ \psi = 1\).
	Since \(\chi \circ \psi\) is an irreducible character of the same degree as \(\chi\), the equality \(\chi = \chi \circ \psi\) must hold as well.
	Therefore, \(R(\psi) = \card{\Fix(\dual{\psi})} \geq 2\) and \cref{prop:ReidemeisterNumber1} implies that \(\psi\) has a non-trivial fixed point.
\end{proof}

We can find another characterisation of \(R(\psi) = 1\) in terms of representation theory.

\begin{prop}	\label{prop:ReidemeisterNumber1IffEquivalentRegularRepresentation}
	Let \(G\) be a finite group and \(\psi \in \End(G)\). Then \(R(\psi) = 1\) if and only if \(T_{\psi}\) is equivalent to the regular representation of \(G\).
\end{prop}
\begin{proof}
	Suppose that \(R(\psi) = 1\). Let \(\chi \in \Irr(G)\) be arbitrary.
	By \cref{lem:ExpressionInproductBitwistedRepresentationIrreducibleCharacter},
	\[
		\inprodarg{\theta_{\psi}}{\chi} = \sum_{\substack{g \in G \\ [\psi(g)] = [g]}} \frac{1}{\card{[\psi(g)]}} \cconj{\chi}(g).
	\]
	Since \(R(\psi) = 1\), \cref{theo:AdoReeNagao} implies that \(1\) is the only \(g \in G\) such that \(\psi(g)\) is conjugate to \(g\).
	Therefore, the sum above reduces to
	\[
		\inprodarg{\theta_{\psi}}{\chi} = \cconj{\chi}(1) = \chi(1).
	\]
	Now, \(\inprodarg{\theta_{\rho}}{\chi} = \chi(1)\) as well, where \(\theta_\rho\) is the character of the regular representation \(\rho\) of \(G\).
	Since \(\chi\) was arbitrary, we conclude that \(\rho\) and \(T_{\psi}\) are equivalent.
	
	Conversely, suppose that \(\rho\) and \(T_{\psi}\) are equivalent.
	Then, by the previous arguments,
	\[
		\theta_{\psi} = \sum_{\chi \in \Irr(G)} \chi(1) \chi.
	\]
	By \cref{prop:ReidemeisterNumberInproductBitwistedRepresentation1},
	\[
		R(\psi) = \inprodarg{\theta_{\psi}}{1} = \sum_{\chi \in \Irr(G)} \chi(1) \inprodarg{\chi}{1}.
	\]
	
	Since \(1\) is an irreducible representation, the inner product \(\inprodarg{\chi}{1}\) is only non-zero when \(\chi = 1\).
	Therefore,
	\[
		R(\psi) = \inprodarg{1}{1} = 1. \qedhere
	\]
\end{proof}

\subsection{Class-preserving endomorphisms}
For various pairs of related endomorphisms, the Reidemeister numbers are equal, as the following folklore result shows:

\begin{prop}
	\label{prop:Reidbijections}
		Let \(G\) be a group, let \(\varphi,\psi \in \End(G)\), let \(\iota \in \Inn(G)\) and let \(\xi \in \Aut(G)\). The following equalities of Reidemeister numbers then hold:
	\begin{align*}
		R(\varphi,\psi) &= R(\psi,\varphi),\\
		R(\varphi,\psi) &= R(\iota \varphi,\psi),\\
		R(\varphi,\psi) &= R(\xi \varphi,\xi \psi).
	\end{align*}
\end{prop}
\begin{proof} 
        Let \(h \in G\) be an element such that \(\iota(g) = hgh^{-1}\) for all \(g \in G\), and consider the following bijections of \(G\):
		\begin{equation*}
		g \mapsto g^{-1},\qquad g \mapsto hg,\qquad g \mapsto \xi(g).
		\end{equation*}
		These induce the following bijections of Reidemeister classes respectively:
		\begin{alignat*}{4}
			&\Reid[\varphi,\psi] \to \Reid[\psi,\varphi]\colon&&[g]_{\varphi,\psi} \mapsto  [g^{-1}]_{\psi,\varphi},\\
			&\Reid[\varphi,\psi] \to \Reid[\iota \varphi,\psi]\colon&&[g]_{\varphi,\psi} \mapsto  [hg]_{\iota \varphi,\psi},\\
			&\Reid[\varphi,\psi] \to \Reid[\xi \varphi,\xi \psi]\colon&&[g]_{\varphi,\psi} \mapsto  [\xi(g)]_{\xi \varphi,\xi \psi}.
		\end{alignat*}
        The equalities of Reidemeister numbers follow from these immediately.
\end{proof}

We extend this with the notion of class-preserving endomorphisms.
Given a group \(G\) and \(\psi \in \End(G)\), we say that \(\psi\) is \emph{class-preserving} if \([\psi(g)] = [g]\) for all \(g \in G\).

All inner automorphisms are class-preserving; however, not all class-preserving endomorphisms are inner, as the following example shows:

\begin{example}
Let \(G\) be the group with presentation \(\langle x,y,z \mid x^8 = y^2 = z^2 = [y,z] = 1, yx = x^3y, zx = x^5z \rangle\), and consider the automorphism \(\psi\) defined by
\[ \psi(x) = x, \quad \psi(y) = x^6y, \quad \psi(z) = z.\]
Any \(g \in G\) for which \(gxg^{-1} = \psi(x) = x\) must itself be a power of \(x\), say \(g = x^a\) with \(0 \leq a \leq 7\). However, if \(x^a y x^{-a} = \psi(y) =  x^6 y\), then \(a \in \{1,5\}\), whereas if \(x^a z x^{-a} = \psi(z) = z\), then \(a\) is even. Thus \(\psi\) is not inner; however, it is class-preserving:
\begin{align*}
\psi(x^a) &= x^a,&\psi(x^a y) &= x^5(x^a y)x^{-5},\\
\psi(x^a z) &= x^a z,&\psi(x^a yz) &= x^3(x^ayz)x^{-3},
\end{align*}
for all \(a \in \{0,\ldots,7\}\).
\end{example}

\begin{prop}	\label{prop:GeneralisationCompositionWithClassPreservingEndomorphismPreservesReidemeisterNumber}
	Let \(G\) be a finite group and \(\phi,\psi \in \End(G)\).
	Suppose \(\xi \in \End(G)\) is class-preserving.
	Then
	\begin{equation*}
		R(\phi, \psi) = R(\xi\phi, \psi)  = R(\phi\xi, \psi)= R(\phi, \xi\psi)= R(\phi,\psi\xi).
	\end{equation*}
\end{prop}
\begin{proof}
	By \cref{cor:SumInproductFormulaReidemeisterNumber}
	\[
		R(\xi \phi, \psi) = \sum_{\chi \in \Irr(G)} \inprodarg{\chi \circ \xi\phi}{\chi \circ \psi}.
	\]
	As \(\chi\) is a character, it is a class function on \(G\).
	So, since \(\xi\) is class-preserving, \(\chi \circ \xi = \chi\).
	Hence,
	\[
		R(\xi \phi, \psi) = \sum_{\chi \in \Irr(G)} \inprodarg{\chi \circ \xi\phi}{\chi \circ \psi} = \sum_{\chi \in \Irr(G)} \inprodarg{\chi \circ \phi}{\chi \circ \psi} = R(\phi, \psi).
	\]
	
	The other equalities are proven similarly.
\end{proof}

This result is, in a sense, surprising. Consider the previous example: while it is the case that \(\# \Conj(G) = R(\psi) = 11\) (recall that \(\Conj(G)\) is the set of (ordinary) conjugacy classes), there are \(2\) conjugacy classes of size \(1\), \(3\) classes of size \(2\), and \(6\) classes of size \(4\), whereas there are \(6\) \(\psi\)-conjugacy classes of size \(2\), and \(5\) classes of size \(4\). Thus, there cannot exist a bijection \(f \colon G \to G\) such that the induced map \({\Conj(G) \to \Reid[\psi] \colon [g] \mapsto [f(g)]_{\psi}}\) is well defined. This is in contrast with \cref{prop:Reidbijections}, where each of the equalities was obtained from such a bijection.

\section{Counting conjugacy classes}
\subsection{Formulae}

Using \cref{theo:ReidemeisterNumberBiTwistedEqualsTraceBilinearForm}, we derive a generalisation of equality (1) and (2) in \cref{theo:AdoReeNagao} for bi-twisted conjugacy.

\biTwistedConjugacyNumberInTermsOfConjugacyClasses
\begin{proof}
For a subset \(X \subseteq G\), we write \(\ind{X}\) for the indicator function of \(X\), that is
\[
	\ind{X} \colon G \to \C \colon g \mapsto \begin{cases}
		1	&	\mbox{ if \(g \in X\)},	\\
		0	&	\mbox{ otherwise}.
	\end{cases}
\]

It is clear that the set \(\{\Delta_{[g]} \mid [g] \in \Conj(G)\}\) forms an orthogonal basis of \(\Class(G)\), where \([g]\) denotes the ordinary conjugacy class of \(g\).
To make it orthonormal, we divide each \(\Delta_{[g]}\) by the square root of its norm, which is given by
\[
	\inprodarg{\Delta_{[g]}}{\Delta_{[g]}} = \frac{1}{\card{G}} \sum_{h \in [g]} 1 = \frac{\card{[g]}}{\card{G}} = \frac{1}{\card{C_{G}(g)}}.
\]
Thus, put \(c_{[g]} \coloneq \sqrt{\card{C_{G}(g)}} \cdot \Delta_{[g]}\).
\cref{theo:ReidemeisterNumberBiTwistedEqualsTraceBilinearForm} applied to \(\{c_{[g]} \mid [g] \in \Conj(G)\}\) states that
\begin{align*}
	R(\phi, \psi)	&= \sum_{[g] \in \Conj(G)} \inprodarg{c_{[g]} \circ \phi}{c_{[g]} \circ \psi}	\\
				&= \sum_{[g] \in \Conj(G)} \card{C_{G}(g)} \cdot \inprodarg{\Delta_{[g]} \circ \phi}{\Delta_{[g]} \circ \psi}.
\end{align*}

Fix \(g \in G\).
For \(h \in G\), we have
\[
	(\Delta_{[g]} \circ \phi)(h) = \begin{cases}
		1	&	\mbox{ if \(\phi(h) \conj g\)},	\\
		0	&	\mbox{ otherwise}.
	\end{cases}
\]
Consequently,
\[
	\Delta_{[g]} \circ \phi = \Delta_{\inv{\phi}([g])},
\]
This implies that
\begin{align*}
	R(\phi, \psi)	&= \sum_{[g] \in \Conj(G)} \card{C_{G}(g)} \cdot \inprodarg{\Delta_{\inv{\phi}([g])}}{\Delta_{\inv{\psi}([g])}}	\\
				&= \sum_{[g] \in \Conj(G)} \frac{\card{C_{G}(g)} \cdot \card{\left(\inv{\phi}([g]) \cap \inv{\psi}([g])\right)}}{\card{G}}	\\
				&= \sum_{[g] \in \Conj(G)} \frac{\card{\left(\inv{\phi}([g]) \cap \inv{\psi}([g])\right)}}{\card{[g]}}. \qedhere
\end{align*}
\end{proof}
Note that for \(\phi = \Id\), we obtain
\[
	R(\psi) = \sum_{[g] \in \Conj(G)} \frac{\card{\left([g] \cap \inv{\psi}([g])\right)}}{\card{[g]}}.
\]
As \(\inv{\psi}([g])\) is a union of conjugacy classes, we find that
\[
	\card{\left([g] \cap \inv{\psi}([g])\right)} = \begin{cases}
		\card{[g]}	&	\mbox{if \([g] \subseteq \inv{\psi}([g])\),}	\\
		0		&	\mbox{otherwise.}
	\end{cases}
\]
In other words, \(\card{\left([g] \cap \inv{\psi}([g])\right)}\) is non-zero if and only if \(\psi([g]) \subseteq [g]\), which is equivalent to \([\psi(g)] = [g]\).
Thus,
\[
	R(\psi) = \sum_{\substack{[g] \in \Conj(G)\\ [\psi(g)] = [g]}} \frac{\card{[g]}}{\card{[g]}} = \sum_{\substack{[g] \in \Conj(G) \\ [\psi(g)] = [g]}} 1,
\]
which is the content of equality of (1) and (2) in \cref{theo:AdoReeNagao}.
\begin{cor}
\label{prop:ReidemeisterCoincidenceNumberEqualsConjugacySum}
Let \(G\) be a finite group and let \(\phi,\psi \in \End(G)\). Then
\begin{equation*}
	R(\phi,\psi) = \sum_{\substack{[g] \in \Conj(G) \\ [\phi(g)] = [\psi(g)]}} \frac{\card{[g]}}{\card{[\phi(g)]}}.
\end{equation*}
\end{cor}

\begin{proof}
As \(\inv{\phi}([g])\) and \(\inv{\psi}([g])\) are both a union of conjugacy classes, we find from \cref{prop:biTwistedConjugacyNumberInTermsOfConjugacyClasses} that
\[
	\card{\left(\inv{\phi}([g]) \cap \inv{\psi}([g])\right)} = \sum_{\substack{[h] \in \Conj(G) \\ [h] \subseteq \inv{\phi}([g]) \cap \inv{\psi}([g])}} \card{[h]}.
\]
For \(h \in G\), the condition \([h] \subseteq \inv{\phi}([g]) \cap \inv{\psi}([g])\) is equivalent to \(\phi([h]) \cup \psi([h]) \subseteq [g]\), which in turn is equivalent to \([\phi(h)] = [\psi(h)] = [g]\).
Therefore,
\begin{align*}
	R(\phi, \psi)	&= \sum_{[g] \in \Conj(G)} \frac{\card{\left(\inv{\phi}([g]) \cap \inv{\psi}([g])\right)}}{\card{[g]}}	\\
				&= \sum_{[g] \in \Conj(G)} \sum_{\substack{[h] \in \Conj(G) \\ [\phi(h)] = [\psi(h)] = [g]}} \frac{\card{[h]}}{\card{[g]}}	\\
				&= \sum_{\substack{[h] \in \Conj(G) \\ [\phi(h)] = [\psi(h)]}} \sum_{\substack{[g] \in \Conj(G) \\ [g] = [\phi(h)] = [\psi(h)]}} \frac{\card{[h]}}{\card{[g]}}	\\
				&= \sum_{\substack{[h] \in \Conj(G) \\ [\phi(h)] = [\psi(h)]}} \frac{\card{[h]}}{\card{[\phi(h)]}}.\qedhere
\end{align*}
\end{proof}

Next, we prove a generalisation of \cref{theo:AdoReeNagao} for single twisted conjugacy:

\reidemeisterNumberEqualsNumberFixedConjugacyClasses

The case \(\phi = \Id\) is then exactly equality of (1) and (2) in \cref{theo:AdoReeNagao}.

\begin{lemma}	\label{lem:SizeStabiliserEqualsSizeTransporter}
Let \(G\) be a group acting on a set \(X\).
Suppose that \(x, y \in X\) are such that \(y \in G \cdot x\).
Then
	\[
		\card{\Stab(x)} = \card{\set{h \in G}{h \cdot x = y}} = \card{\set{h \in G}{h \cdot y = x}}.
	\]
\end{lemma}
\begin{proof}
	Let \(g_{0} \in G\) be such that \(g_{0} \cdot x = y\).
	Then for all \(g \in G\), we have
	\[
		g \cdot x = y \iff (\inv{g}_{0} g) \cdot x = x.
	\]
	Thus, the map \(g \mapsto \inv{g}_{0} g\) defines a bijection between \(\set{g \in G}{g \cdot x = y}\) and \(\Stab(x)\).
	
	For the second equality, note that, for all \(g \in G\), \(g \cdot x = y\) if and only if \(\inv{g} \cdot y = x\).
	Hence, the map \(g \mapsto \inv{g}\) defines a bijection between \(\set{g \in G}{g \cdot x = y}\) and \(\set{g \in G}{g \cdot y = x}\).
\end{proof}

\begin{proof}[Proof of \cref{theo:ReidemeisterNumberEqualsNumberFixedConjugacyClasses}]
	First, we have to prove that \(\Xi([g]_{\phi})\) is independent of the chosen representative.
Suppose that \(g \Rconj{\phi} h\) for some \(g, h \in G\).
Write \(g = x h \inv{\phi(x)}\) for some \(x \in G\).
Then
	\[
		\psi(g) = \psi(x) \psi(h) \psi\left(\inv{\phi(x)}\right) = \psi(x) \psi(h) \inv{\phi(\psi(x))},
	\]
	since \(\phi\) and \(\psi\) commute.
Therefore, \(\psi(g) \Rconj{\phi} \psi(h)\).
	
	Next, we prove that \(R(\psi) = \card{\Fix(\Xi)}\).
	We base our argument on the one Nagao gives for \cref{theo:AdoReeNagao} \cite[Lemma~2]{Nagao62}.
	
	We count the number of solutions in \(G \times G\) of the equation
	\begin{equation}	\label{eq:alteredTwistedConjugacyEquation1}
		x = g \phi(x) \inv{\psi(g)}
	\end{equation}
	or, equivalently, of
	\begin{equation}	\label{eq:alteredTwistedConjugacyEquation2}
		g = x \psi(g) \inv{\phi(x)}.
	\end{equation}
	
	For \eqref{eq:alteredTwistedConjugacyEquation1}, let \(x \in G\) be fixed.
	By assumption, \(x \Rconj{\psi} \phi(x)\), so there is at least one solution.
	\cref{lem:SizeStabiliserEqualsSizeTransporter} then implies that there are exactly \(\card{\Stab_{\psi}(x)}\) solutions.
	Thus, the number of solutions is given by
	\begin{equation}	\label{eq:solutionsToAlteredTwistedConjugacyEquation1}
		\sum_{x \in G} \card{\Stab_{\psi}(x)} = \card{G} \cdot \sum_{x \in G} \frac{1}{\card{[x]_{\psi}}} = \card{G} \cdot R(\psi).
	\end{equation}
	where we used the orbit-stabiliser theorem and the fact that for each \(x \in G\), the class \([x]_{\psi}\) is counted exactly \(\card{[x]_{\psi}}\) times in the last summation.
	
	For \eqref{eq:alteredTwistedConjugacyEquation2}, let \(g \in G\) be fixed.
	For there to be a solution, \(g\) and \(\psi(g)\) must be \(\phi\)-conjugate.
	If there is a solution, \cref{lem:SizeStabiliserEqualsSizeTransporter} implies that there are exactly \(\card{\Stab_{\phi}(g)}\) solutions.
	Thus, the number of solutions is given by
	
	\begin{equation}	\label{eq:solutionsToAlteredTwistedConjugacyEquation2}
		\sum_{\substack{g \in G \\ g \Rconj{\phi} \psi(g)}} \card{\Stab_{\phi}(g)} = \card{G} \cdot \sum_{\substack{g \in G \\ g \Rconj{\phi} \psi(g)}}  \frac{1}{\card{[g]_{\phi}}} = \card{G} \cdot \card{\Fix(\Xi)},
	\end{equation}
	where we again used the orbit-stabiliser theorem.
	For the last equality, we have to prove that given a \(g \in G\) such that \(g \Rconj{\phi} \psi(g)\), each \(x \in [g]_{\phi}\) also satisfies \(x \Rconj{\phi} \psi(x)\).
	Earlier, we argued that \(g \Rconj{\phi} x\) implies \(\psi(g) \Rconj{\phi} \psi(x)\).
	Therefore,
	\[
		x \Rconj{\phi} g \Rconj{\phi} \psi(g) \Rconj{\phi} \psi(x).
	\]
	
	Finally, from \eqref{eq:solutionsToAlteredTwistedConjugacyEquation1} and \eqref{eq:solutionsToAlteredTwistedConjugacyEquation2}, it follows that \(R(\psi) = \card{\Fix(\Xi)}\).

	To argue that \(\phi = \psi^{k}\) satisfies the conditions of the theorem for all \(k \geq 1\), first note that clearly \(\psi \circ \psi^{k} = \psi^{k} \circ \psi\).
Next, observe that \(g = g \psi(g) \inv{\psi(g)}\) for all \(g \in G\).
Hence, \(g \Rconj{\psi} \psi(g)\) for all \(g \in G\).
By induction and transitivity of \(\psi\)-twisted conjugacy, we obtain that \(g \Rconj{\psi} \psi^{k}(g)\) for all \(g \in G\).
\end{proof}

\subsection{Inequalities}
We start with an inequality for Reidemeister numbers of powers of endomorphisms.
\begin{prop}	\label{prop:ReidemeisterNumberOfPowerOfEndomorphismFiniteGroupIsGreater}
	Let \(G\) be a finite group and \(\psi \in \End(G)\).
Let \(k \geq 0\) be an integer.
Then \(R(\psi^{k}) \geq R(\psi)\).
In particular, if \(\psi \in \Aut(G)\) and \(k\) is coprime with the order of \(\psi\), then \(R(\psi^{k}) = R(\psi)\).
\end{prop}
\begin{proof}
If we put \(\phi = \psi^{k}\) in \cref{theo:ReidemeisterNumberEqualsNumberFixedConjugacyClasses}, we immediately get
\[
	R(\psi) = \card{\Fix(\Xi)} \leq \card{\Reid[\psi^{k}]} = R(\psi^{k}).
\]
Now, let \(\psi \in \Aut(G)\) and let \(n\) be its order.
Let \(k \geq 0\) be coprime with \(n\).
We already have \(R(\psi^{k}) \geq R(\psi)\).
Since \(k\) and \(n\) are coprime, there exist \(a, b \in \Z\) such that \(ak + bn = 1\).
We may assume that \(a \geq 1\).
Then \(\psi = \psi^{ak + bn} = \psi^{ak}\).
Since \(a \geq 1\), we have \(R(\psi) = R(\psi^{ak}) \geq R(\psi^{k})\).
Consequently, \(R(\psi) = R(\psi^{k})\).
\end{proof}

Next, for a central extension \(1 \to C \to G \to G / C \to 1\) and \(\phi,\psi \in \End(G)\) such that \(\phi(C),\psi(C) \leq C\), we can find both an upper and a lower bound on the Reidemeister number \(R(\phi,\psi)\). We start with the upper bound, which does not require \(G\) to be a finite group.

\begin{prop}
Let \(G\) be a group and \(C \leq Z(G)\) a central subgroup. Let \(\phi,\psi \in \End(G)\) be such that \(\phi(C),\psi(C) \leq C\). Let \(\restr{\phi}{C}\) and \(\restr{\psi}{C}\) denote the restricted endomorphisms on \(C\), and \(\bar{\phi},\bar{\psi}\) the induced endomorphisms on \(G/C\). Then \(R(\phi,\psi) \leq R(\restr{\phi}{C},\restr{\psi}{C})R(\bar{\phi},\bar{\psi})\).
\end{prop}

This inequality follows from the arguments presented in Section 2 of \cite{GoncalvesWong05}, where the authors actually obtain equality under additional assumptions. We give an elementary proof below to keep the paper self-contained.

\begin{proof}
We may assume both \(R(\restr{\phi}{C},\restr{\psi}{C})\) and \(R(\bar{\phi},\bar{\psi})\) are finite, otherwise the statement is vacuously true. Choose representatives \(c_1, \ldots, c_m \in C\) of the \((\restr{\phi}{C},\restr{\psi}{C})\)-conjugacy classes and representatives \(g_1C, \ldots, g_nC \in G/C\) of the \((\bar{\phi},\bar{\psi})\)-conjugacy classes.

Let \(g \in G\) arbitrarily. Then for some \(j \in \{1, \ldots, n\}\) and \(hC \in G/C\), we have
\[
    gC = \bar{\phi}(hC)g_jC\bar{\psi}(hC)^{-1} = \phi(h)g_j\psi(h)^{-1}C.
\]
Hence, there exists a \(c \in C\) such that
\begin{equation}
\label{eq:uppboundeq1}
    g = \phi(h)g_j\psi(h)^{-1}c.
\end{equation}
Working inside \(C\), for some \(i \in \{1,\ldots,m\}\) and \(k \in C\) we have
\begin{equation}
\label{eq:uppboundeq2}
    c = \phi(k)c_i\psi(k)^{-1}.
\end{equation}
Combining \eqref{eq:uppboundeq1} and \eqref{eq:uppboundeq2}, and using that \(C\) is central, we get
\[
    g = \phi(h)g_j\psi(h)^{-1}\phi(k)c_i\psi(k)^{-1} = \phi(kh)c_ig_j\psi(kh)^{-1}.
\]
Thus, each \(g \in G\) is \((\phi,\psi)\)-conjugate to some element of the form \(c_ig_j\), and hence there are at most
\(mn = R(\restr{\phi}{C},\restr{\psi}{C})R(\bar{\phi},\bar{\psi})\) \((\phi,\psi)\)-conjugacy classes.
\end{proof}

For finite groups, we have an additional lower bound.

\begin{prop}	\label{prop:ReidemeisterNumberOnCentreIsLowerBoundFiniteGroups}
	Let \(G\) be a finite group and \(C \leq Z(G)\) a central subgroup. Let \(\phi,\psi \in \End(G)\) be such that \(\phi(C),\psi(C) \leq C\). Let \(\restr{\phi}{C}\) and \(\restr{\psi}{C}\) denote the restricted endomorphisms on \(C\). Then \(R(\restr{\phi}{C},\restr{\psi}{C}) \leq R(\phi,\psi)\).
\end{prop}
\begin{proof}
	Since \(C\) is central, the \(C\)-conjugacy class of \(c \in C\) is a singleton and coincides with the \(G\)-conjugacy class. From \cref{prop:biTwistedConjugacyNumberInTermsOfConjugacyClasses} we get that
	\begin{align*}
		R(\phi,\psi)	&= \sum_{[g] \in \Conj(G)} \frac{\card{(\inv{\phi}([g]) \cap \inv{\psi}([g]))}}{\card{[g]}}	\\
					& \geq \sum_{[g] \in \Conj(C)} \frac{\card{(\inv{\phi}([g]) \cap \inv{\psi}([g]))}}{\card{[g]}}	\\
					& \geq \sum_{[g] \in \Conj(C)} \frac{\card{(\invrestr{\phi}{C}([g]) \cap \invrestr{\psi}{C}([g]))}}{\card{[g]}}	\\
					& = R(\restr{\varphi}{C},\restr{\psi}{C}). \qedhere
	\end{align*}
\end{proof}

Finally, for single twisted conjugacy, \cref{theo:AdoReeNagao} clearly implies that Reidemeister numbers \(R(\psi)\) are bounded above by the number of ordinary conjugacy classes of the group.
At a first glance, for bi-twisted conjugacy, no such upper bound exists.
Indeed, if \(\tau\) is the endomorphism that maps everything to the identity, then \(R(\tau,\tau) = \card{G}\).

However, upon taking a closer look, we note that this maximal Reidemeister number is an isolated case.
We show that there exists a \(k \leq \frac{3}{4}\card{G}\) such that no pair of endomorphisms \((\phi,\psi)\) can have Reidemeister number \(\card{G} > R(\phi,\psi) > k\). To prove this, we will require the following technical lemma.

\begin{lemma}
\label{lem:Hconjugacybound}
    Let \(G\) be a non-abelian finite group and let \(H\) be a non-central subgroup of \(G\). Let \(p\) be the smallest prime divisor of \(\card{G}\), and let \(q\) be either the second smallest prime factor (if \(p^3 \nmid \card{G}\)) or equal to \(p\) (if \(p^3 \mid \card{G}\)). Let \(k_H(G)\) denote the number of \(H\)-conjugacy classes in \(G\). Then
    \[
     k_H(G)  \leq \frac{p+q-1}{pq} \card{G}.
    \]
\end{lemma}
\begin{proof}
Suppose there exist \(a,b \in \N\) such that \([G:C_G(H)] \geq a\), and such that \([H:C_H(g)] \geq b\) for all \(g \notin C_{G}(H)\). Then
\begin{align*}
    k_H(G)
    &= \card{C_G(H)} + \sum_{g \notin C_G(H)} \frac{1}{[H:C_H(g)]}\\
    &\leq \card{C_G(H)} + \frac{\card{G} - \card{C_G(H)}}{b}\\
    &= \frac{(b-1)\card{C_G(H)} + \card{G}}{b}\\
    &\leq \frac{a + b - 1}{ab}\card{G}.
    \end{align*}
    It remains to prove that we can always pick \((a,b)\) to be either \((p,q)\) or \((q,p)\). Since \(H\) is non-central, \(C_G(H) \neq G\) and therefore \([G:C_G(H)] \geq p\). Also, if \(g \notin C_G(H)\), then \(C_H(g)\) is a proper subgroup of \(H\), thus \([H:C_H(g)] \geq p\).

    If \(p^3 \mid \card{G}\), and therefore \(p = q\), we are done. If \(p^3 \nmid \card{G}\), and therefore \(p < q\), we need to show that one of these indices is actually greater than or equal to \(q\). Let \(H_p\) be a Sylow \(p\)-subgroup of \(H\), and let \(G_p\) be a Sylow \(p\)-subgroup of \(G\) containing \(H_p\). We consider two cases:
    \begin{enumerate}
        \item \(H_p \leq Z(G)\). Then \(H_p \leq C_H(g)\) for every \(g \in G\), hence \(p\) does not divide \([H:C_H(g)]\). So if \(g \notin C_G(H)\), the index \([H:C_H(g)]\) is at least \(q\).
        \item \(H_p \not\leq Z(G)\). Since \(k_{H}(G) \leq k_{H_p}(G)\), we may assume that \(H = H_p\). Because \(G_p\) has order \(p\) or \(p^2\), it is abelian. Hence, \(H_p \leq Z(G_p)\) and thus \(G_p \leq C_G(H_p)\). Then \(p\) does not divide \([G:C_G(H_p)]\), and thus this index is at least \(q\).
    \end{enumerate}
    Thus, we can always pick \((a,b)\) to be either \((p,q)\) or \((q,p)\).
\end{proof}

\upperboundBiTwistedConjugacyNumbers
\begin{proof}
    We only prove that (1) implies (2), since the implications (2) \(\implies\) (3) and (3) \(\implies\) (1) are both obvious.
    From \cref{prop:biTwistedConjugacyNumberInTermsOfConjugacyClasses} we know that \(R(\phi,\phi) \geq R(\phi,\psi)\). The \((\phi,\phi)\)-conjugacy action coincides with the \(\im(\phi)\)-conjugacy action, and thus \(R(\phi,\phi) = k_{\im(\phi)}(G)\). By \cref{lem:Hconjugacybound} this means that \(\im(\phi) \leq Z(G)\), and similarly, we obtain that \(\im(\psi) \leq Z(G)\).

    Therefore, the map \(\delta \colon G \to G \colon g \mapsto \phi(g)\psi(g)^{-1}\) is an endomorphism, and \(R(\phi,\psi) = [ G : \im(\delta) ]\). From the original inequality then follows that \(\card{\im(\delta)} < \frac{pq}{p+q-1} < p\), hence \(\delta\) must be the trivial map, i.e.\@ \(\phi = \psi\).
\end{proof}

This bound is optimal, as the following examples show.

\begin{example}
    Let \(G\) be a non-abelian group of order \(pq\), with \(p < q\) primes. Then \(G\) is a semi-direct product \(C_q \rtimes C_p\) of two cyclic groups, so we can choose \(x,y \in G\) such that \(C_q = \langle x \rangle\) and \(C_p = \langle y \rangle\).
    The conjugacy classes of \(G\) are:
    \begin{itemize}
        \item one central class, \([1]\),
        \item \(p-1\) classes of size \(q\): \([y], \ldots, [y^{p-1}]\),
        \item \(\frac{q-1}{p}\) classes of size \(p\): \([x], \ldots, [x^{\frac{q-1}{p}}]\).
    \end{itemize}
    Define the endomorphism \(\phi\) as
    \[
        \phi(x) = 1, \qquad \phi(y) = y.
    \]
    Applying the summation formula from \cref{prop:ReidemeisterCoincidenceNumberEqualsConjugacySum} gives
    \[
        R(\varphi,\varphi) = \frac{1}{1} + (p-1)\frac{q}{q} + \frac{q-1}{p}\frac{p}{1} = p+q-1 = \frac{p+q-1}{pq} \card{G}.
    \]
    So, the bound from \cref{prop:upperboundBiTwistedConjugacyNumbers} is indeed attained.
\end{example}

\begin{example}
	Let \(p\) be a prime number and \(e \in \{1, 2\}\).
	Let \(G_{p,1,e}\) be the group given by the presentation
	\[
		\grppres{x, y, z}{x^p = y^p= z^{e - 1}, xy = yxz, z^p = [x, z] = [y, z] = 1},
	\]
	i.e.\@ the extraspecial group of order \(p^3\) with, if \(p > 2\), exponent \(p^e\). If \(p = 2\), the exponent of \(G_{2,1,e}\) is always \(4\).
    
    An extraspecial group of order \(p^{2n+1}\) is the central product of \(n\) extraspecial groups of order \(p^3\). For every \(n\) there are only two such groups up to isomorphism. We define \(G_{p,n,1}\) to be the central product of \(n\) copies of \(G_{p,1,1}\), and \(G_{p,n,2}\) the central product of \(n-1\) copies of \(G_{p,1,1}\) and a single copy of \(G_{p,1,2}\).
    
    Fix a triple \((p,n,e)\) with \(p\) odd. Let \(G_{p,n,e}\) be generated by \(x_1\), \(y_1\), \(x_2\), \(y_2\), \ldots, \(x_n\), \(y_n\), \(z\) where \(x_i\), \(y_i\), \(z\) are the generators of the \(i\)-th factor in the central product. Every non-central element has centraliser of order \(p^{2n}\) and hence conjugacy class of size \(p\). Thus, there are \(p^{2n}-1\) distinct conjugacy classes of non-central elements, and they are given by
    \[
        \set{ [x_1^{a_1} y_1^{b_1} \dots x_n^{a_n} y_n^{b_n}] }{ 0 \leq a_i,b_j < p\,;\ (a_1,b_1,\dots,a_n,b_n) \neq (0,\dots,0)}.
    \]
    Let \(\phi\) be the endomorphism defined by
    \begin{align*}
        &\phi(x_1) = x_1 y_1^{-1},&&  \phi(x_i) = 1 \quad (2 \leq i \leq n),\\
        &\phi(y_j) = 1 \quad (1 \leq j \leq n), && \phi(z) = 1.
    \end{align*}
    This endomorphism maps:
    \begin{itemize}
        \item \((p-1)p^{2n-1}\) of the non-central conjugacy classes (i.e.\@ those with \(a_1 \neq 0\)) to another non-central conjugacy class,
        \item \(p^{2n-1}-1\) of the non-central conjugacy classes (i.e.\@ those with \(a_1 = 0\)) to a central conjugacy class,
        \item all \(p\) central conjugacy classes to another central conjugacy class.
    \end{itemize}
    Applying the summation formula from \cref{prop:ReidemeisterCoincidenceNumberEqualsConjugacySum} gives
    \[
        R(\varphi,\varphi) = (p-1)p^{2n-1} \frac{p}{p} + (p^{2n-1}-1) \frac{p}{1} + p\frac{1}{1} = 2p^{2n}-p^{2n-1} = \frac{2p-1}{p^2} \card{G_{p,n,e}},
    \]
    so indeed the bound from \cref{prop:upperboundBiTwistedConjugacyNumbers} is attained.
    
    For \(p = 2\) and either \(e=1\) or \(n \geq 2\), we can repeat the above process, although we have to modify the definition of \(\phi\) by setting \(\phi(x_1) = x_1y_1^2\). For the group \(G_{2,1,2} \cong Q_8\), however, no endomorphism \(\phi\) with \(R(\phi,\phi) = 6\) exists.
\end{example}

\begin{quest}
	Which finite groups \(G\) admit endomorphisms \(\phi, \psi\) such that \(R(\phi, \psi) = \frac{p+q - 1}{pq} \cdot \card{G}\), where \(p\) and \(q\) are defined as in \cref{prop:upperboundBiTwistedConjugacyNumbers}?
\end{quest}
\begin{remark}
	For a finite abelian group \(A\) and \(\phi, \psi \in \End(A)\), the optimal bound for non-trivial Reidemeister numbers is \(R(\phi, \psi) \leq \frac{1}{p} \cdot \card{A}\), with \(p\) the smallest prime dividing \(\card{A}\).
	Indeed, we have \(R(\phi, \psi) = [A \colon \im(\phi - \psi)]\), since
	\[
		x \Rconj{\phi, \psi} y \iff \exists a \in A  \colon x = \phi(a) + y - \psi(a) \iff x - y \in \im(\phi - \psi).
	\]
	Thus, \(R(\phi, \psi)\) divides \(\card{A}\).
	The largest non-trivial divisor of \(\card{A}\) is \(\frac{\card{A}}{p}\), and this number can be achieved.
	Write \(A = C_{p^n} \times B\) for some \(n \geq 1\) and some subgroup \(B\) of \(A\).
	Let \(x\) be a generator of \(C_{p^n}\) and define the endomorphism \(\phi \in \End(A)\) by \(\phi(x^i, b) = (x^{i(p + 1)}, b)\).
	Then \(\im(\phi - \Id) = \grpgen{x^p}\) and thus
	\[
		R(\phi, \Id) = [A : \grpgen{x^p}] = [C_{p^n} : \grpgen{x^p}] \cdot \card{B} = p^{n - 1} \cdot \card{B} = \frac{\card{A}}{p}.
	\]
	
	This lower bound is slightly lower than the one for non-abelian groups, since for primes \(p \leq q\), we have \(p + q - 1 > q\) and thus
	\[
		\frac{p + q - 1}{pq} > \frac{q}{pq} = \frac{1}{p}.
	\]

\end{remark}

\section*{Acknowledgements}
We thank the anonymous referee for pointing us to the results of I.\ Ado, R.\ Ree and H.\ Nagao, and to the various other places where twisted conjugacy in finite groups also arises.

\printbibliography

@article{Ado55,
	author = {Ado, Igor Dmitrievich},
	journal = {Matematicheskii Sbornik},
	number = {1},
	pages = {25--30},
	title = {{On the theory of linear representations of finite groups}},
	volume = {36(78)},
	year = {1955}}

@article{BardakovNasybullovNeshchadim13,
	author = {Bardakov, V. G. and Nasybullov, Timur and Neshchadim, M. V.},
	doi = {10.1134/S0037446613010023},
	journal = {Siberian Mathematical Journal},
	number = {1},
	pages = {10--21},
	title = {{Twisted conjugacy classes of the unit element}},
	volume = {54},
	year = {2013}}

@article{Chen22,
	author = {Chen, Zhe},
	doi = {10.1016/j.jalgebra.2022.06.029},
	journal = {Journal of Algebra},
	pages = {718--733},
	title = {{Twisting operators and centralisers of Lie type groups over local rings}},
	volume = {609},
	year = {2022}}

@article{DeligneLusztig76,
	author = {Deligne, Pierre and Lusztig, George},
	doi = {10.2307/1971021},
	journal = {Annals of Mathematics},
	number = {1},
	pages = {103--161},
	title = {{Representations of Reductive Groups Over Finite Fields}},
	volume = {103},
	year = {1976}}

@article{Deshpande16,
	author = {Deshpande, Tanmay},
	doi = {10.1112/S0010437X16007429},
	journal = {Compositio Mathematica},
	number = {8},
	pages = {1697--1724},
	title = {{Shintani descent for algebraic groups and almost characters of unipotent groups}},
	volume = {152},
	year = {2016}}

@article{FelshtynHill94,
	author = {Fel'shtyn, Alexander Leopol'dovich and Hill, Richard},
	doi = {10.1007/BF00961408},
	journal = {K-Theory},
	number = {4},
	pages = {367--393},
	title = {{The Reidemeister Zeta Function with Applications to Nielsen Theory and a Connection with Reidemeister Torsion}},
	volume = {8},
	year = {1994}}

@article{FelshtynKlopsch22,
	author = {Fel'shtyn, Alexander Leopol'dovich and Klopsch, Benjamin},
	doi = {10.1016/j.indag.2022.02.004},
	journal = {Indagationes Mathematicae},
	number = {4},
	pages = {753--767},
	title = {{P\'{o}lya-Carlson dichotomy for coincidence Reidemeister zeta functions via profinite completions}},
	volume = {33},
	year = {2022}}

@article{Formanek71,
	author = {Formanek, Edward},
	doi = {10.1090/S0002-9939-1971-0291312-5},
	journal = {Proceedings of the American Mathematical Society},
	number = {1},
	pages = {73--74},
	title = {{The conjugation representation and fusionless extensions}},
	volume = {30},
	year = {1971}}

@article{Frame47,
	author = {Frame, J. Sutherland},
	doi = {10.1090/S0002-9904-1947-08839-X},
	journal = {Bulletin of the American Mathematical Society},
	number = {6},
	pages = {584--589},
	title = {{On the reduction of the conjugating representation of a finite group}},
	volume = {53},
	year = {1947}}

@article{GoncalvesNasybullov18,
	author = {Gon\c{c}alves, Daciberg Lima and Nasybullov, Timur},
	doi = {10.1080/00927872.2018.1498873},
	journal = {Communications in Algebra},
	number = {3},
	pages = {930--944},
	title = {{On groups where the twisted conjugacy class of the unit element is a subgroup}},
	volume = {47},
	year = {2018}}

@article{GoncalvesWong05,
	author = {Gon\c{c}alves, Daciberg Lima and Wong, Peter N.},
	doi = {10.1515/form.2005.17.2.297},
	journal = {Forum Mathematicum},
	number = {2},
	pages = {297--313},
	title = {{Homogeneous spaces in coincidence theory II}},
	volume = {17},
	year = {2005}}

@article{GorensteinHerstein61,
	author = {Gorenstein, Daniel and Herstein, I. N.},
	doi = {10.2307/2372721},
	journal = {American Journal of Mathematics},
	number = {1},
	pages = {71--78},
	title = {{Finite Groups Admitting a Fixed-Point Free Automorphism of Order \(4\)}},
	volume = {83},
	year = {1961}}

@book{Jiang83,
	address = {Providence, Rhode Island},
	author = {Jiang, Boju},
	doi = {10.1090/conm/014},
	publisher = {American Mathematical Society},
	series = {Contemporary Mathematics},
	title = {{Lectures on Nielsen Fixed Point Theory}},
	volume = {14},
	year = {1983}}

@book{Karpilovsky92,
	author = {Karpilovsky, Gregrory},
	publisher = {Elsevier Science},
	series = {Group Representations},
	title = {{Introduction to Group Representations and Characters}},
	volume = {1},
	part = {B},
	year = {1992}}

@article{Lusztig14,
	author = {Lusztig, George},
	doi = {10.1090/S1088-4165-2014-00455-2},
	journal = {Representation Theory of the American Mathematical Society},
	number = {8},
	pages = {223--277},
	title = {{Distinguished Conjugacy Classes and Elliptic Weyl Group Elements}},
	volume = {18},
	year = {2014}}

@article{Nagao62,
	author = {Nagao, Hirosi},
	journal = {Journal of Mathematics, Osaka City University},
	number = {1},
	pages = {35--38},
	title = {{On a conjecture of Brauer for \(p\)-solvable groups}},
	volume = {13},
	year = {1962}}

@article{Nicotera24,
	author = {Nicotera, Chiara},
	doi = {10.1007/s00013-024-02025-6},
	journal = {Archiv der Mathematik},
	number = {3},
	pages = {225--232},
	title = {{On finite groups in which the twisted conjugacy classes of the unit element are subgroups}},
	volume = {123},
	year = {2024}}

@article{Ree59,
	author = {Ree, Rimhak},
	doi = {10.1215/ijm/1255455264},
	journal = {Illinois Journal of Mathematics},
	number = {3},
	pages = {440--444},
	title = {{On Generalized Conjugate Classes in a Finite Group}},
	volume = {3},
	year = {1959}}

@article{Romankov16,
	author = {Roman'kov, Vitalii Anatol'evich},
	doi = {10.17377/semi.2016.13.056},
	journal = {Sibirskie {\`E}lektronnye Matematicheskie Izvestiya},
	pages = {716--725},
	title = {{On solvability of equations with endomorphisms in nilpotent groups}},
	volume = {13},
	year = {2016}}

@article{Roth71,
	author = {Roth, Richard Lewis},
	doi = {10.2140/pjm.1971.36.515},
	journal = {Pacific Journal of Mathematics},
	number = {2},
	pages = {515--521},
	title = {{On the Conjugating Representation of a Finite Group}},
	volume = {36},
	year = {1971}}

@article{Rowley95,
	author = {Rowley, Peter},
	doi = {10.1006/jabr.1995.1148},
	journal = {Journal of Algebra},
	number = {2},
	pages = {724--727},
	title = {{Finite Groups Admitting a Fixed-Point-Free Automorphism Group}},
	volume = {174},
	year = {1995}}

@misc{Senden21a,
	archiveprefix = {arXiv},
	author = {Senden, Pieter},
	doi = {10.48550/arXiv.2109.12892},
	eprint = {2109.12892},
	primaryclass = {math.GR},
	title = {{The Reidemeister spectrum of split metacyclic groups}},
	year = {2021}}

@phdthesis{Senden23,
	author = {Senden, Pieter},
	school = {KU Leuven},
	title = {{How does the structure of a group determine its Reidemeister spectrum?}},
	year = {2023}}

@article{Senden23a,
	author = {Senden, Pieter},
	doi = {10.1017/S0013091523000500},
	journal = {Proceedings of the Edinburgh Mathematical Society},
	number = {4},
	pages = {940--959},
	title = {{The Reidemeister spectrum of finite abelian groups}},
	volume = {66},
	year = {2023}}

@book{Serre77,
	address = {New York},
	author = {Serre, Jean-Pierre},
	doi = {10.1007/978-1-4684-9458-7},
	place = {New York},
	publisher = {Springer, New York},
	series = {Graduate Texts in Mathematics},
	title = {{Linear Representations of Finite Groups}},
	volume = {42},
	year = {1977}}

@article{Shintani76,
	author = {Shintani, Takuro},
	doi = {10.2969/jmsj/02820396},
	journal = {Journal of the Mathematical Society of Japan},
	number = {2},
	pages = {396--414},
	title = {{Two remarks on irreducible characters of finite general linear groups}},
	volume = {28},
	year = {1976}}

@inproceedings{Shoji85,
	author = {Shoji, Toshiaki},
	booktitle = {Algebraic Groups and Related Topics},
	doi = {10.2969/aspm/00610207},
	editor = {Hotta, R.},
	pages = {207--229},
	title = {{Some Generalization of Asai's Result for Classical Groups}},
	year = {1985}}

@article{Shoji87,
	author = {Shoji, Toshiaki},
	doi = {10.15083/00039463},
	journal = {Journal of the Faculty of Science. University of Tokyo. Section IA. Mathematics},
	number = {3},
	pages = {599--653},
	title = {{Shintani descent for exceptional groups over a finite field}},
	volume = {34},
	year = {1987}}

@article{Shoji92,
	author = {Shoji, Toshiaki},
	doi = {10.1016/0021-8693(92)90113-Z},
	journal = {Journal of Algebra},
	number = {2},
	pages = {468--524},
	title = {{Shintani Descent for Algebraic Groups over a Finite Field, I}},
	volume = {145},
	year = {1992}}

@article{ShumyatskyTamarozzi02,
	author = {Shumyatsky, Pavel and Tamarozzi, Antonio},
	doi = {10.1081/AGB-120003992},
	journal = {Communications in Algebra},
	number = {6},
	pages = {2837--2842},
	title = {{On finite groups with fixed-point free automorphisms}},
	volume = {30},
	year = {2002}}

@article{Solomon61,
	author = {Solomon, Louis},
	doi = {10.1090/S0002-9939-1961-0132783-8},
	journal = {Proceedings of the American Mathematical Society},
	number = {6},
	pages = {962--963},
	title = {{On the sum of the elements in the character table of a finite group.}},
	volume = {12},
	year = {1961}}

@article{Tertooy25,
	author = {Tertooy, Sam},
	doi = {10.1515/jgth-2024-0056},
	journal = {Journal of Group Theory},
	number = {3},
	pages = {563--583},
	title = {{Twisted conjugacy and separability}},
	volume = {28},
	year = {2025}}

@article{Tertooy25a,
	author = {Tertooy, Sam},
	doi = {10.1016/j.topol.2024.109089},
	journal = {Topology and its Applications},
	eid = {109089},
	title = {{Extreme Reidemeister spectra of finite groups}},
	volume = {359},
	year = {2025}}

@article{Thompson59,
	author = {Thompson, John},
	doi = {10.1073/pnas.45.4.578},
	journal = {Proceedings of the National Academy of Sciences of the United States of America},
	number = {4},
	pages = {578--581},
	title = {{Finite Groups with Fixed-Point Free Automorphisms of Prime Order}},
	volume = {45},
	year = {1959}}

\end{document}